\documentclass[10pt]{article}
\usepackage{supertabular}
\usepackage[dvips]{graphics}
\usepackage{amsmath}
\usepackage{amsfonts}
\usepackage{latexsym}
\usepackage{wrapfig}
\usepackage{amssymb, amsthm}
\usepackage{amsfonts}
\usepackage{enumerate}

\newtheorem{dummy}{anything}
\newtheorem{theorem}[dummy]{Theorem}
\newtheorem{lemma}{Lemma}
\newtheorem{proposition}{Proposition}
\newtheorem{corollary}{Corollary}
\newtheorem{definition}{Definition}

\newcommand{\tc}{{\rm\bf {TC}}}

\newcommand{\pcirc}{\kern .7pt {\scriptstyle \circ} \kern 1pt}

\newcommand{\R}{{\mathbf R}}

\newcommand{\Z}{{\mathbf Z}}

\newcommand{\vol}{{\rm {vol}}}

\newcommand{\hfl}[2]{\smash{\mathop{\hbox to 1 truecm{\kern %
3pt\rightarrowfill\kern 3pt}}%
\limits^{\scriptstyle#1}_{\scriptstyle#2}}}
\newcommand{\cqfd}{\unskip\kern 6pt\penalty 500%
\raise -2pt\hbox{\vrule\vbox to10pt{\hrule width %
4pt\vfill\hrule}\vrule}\smallskip}

\title{Topology of random linkages}
\author{Michael Farber\footnote{The research was supported by a grant from the Royal Society and from the UK
  Engineering and Physical Sciences Research Council.}
\\Department of Mathematical Sciences\\ University of Durham, UK}
\date{August 7, 2007}

\begin{document}
\maketitle 

\begin{abstract}
Betti numbers of configuration spaces of mechanical linkages (known also as polygon spaces)
depend on a large number of parameters -- the lengths of the bars of the linkage.
Motivated by applications in topological robotics, statistical shape theory and molecular biology,
we view these lengths as random variables and study asymptotic values of the average Betti numbers as the number of links $n$ tends to infinity.
We establish a surprising fact that for a reasonably ample class of sequences of probability measures the asymptotic values of the average Betti numbers are
independent of the choice of the measure. The main results of the paper apply to planar linkages as well as for linkages in $\R^3$.
We also prove results about higher moments of Betti numbers.
\end{abstract}

\section{Introduction}

Configuration spaces of mechanical systems which appear in
topological robotics depend typically on many
parameters which are only partially known and can be viewed as
random variables. The configuration space,
which depends essentially on the values of the parameters, can be viewed in such
a situation as a {\it random topological space}.  Betti numbers of the configuration space
are then random functions and information about their mathematical expectations and other statistical characteristics
may have practical importance in various applications.

Interesting examples of such random topological
spaces are provided by configuration spaces of mechanical linkages, the main object of study in this paper.
A {\it linkage} is a simple mechanism consisting of $n$ bars in $\R^3$ having fixed lengths $l_1, \dots, l_n$ which are cyclically connected by
revolving joints forming a closed polygonal chain. Angles between bars of the linkage may vary, the only condition is that the links do not become
disconnected from each other.
 \begin{figure}
\begin{center}
\resizebox{5.5cm}{4cm}
{\includegraphics[74,501][498,825]{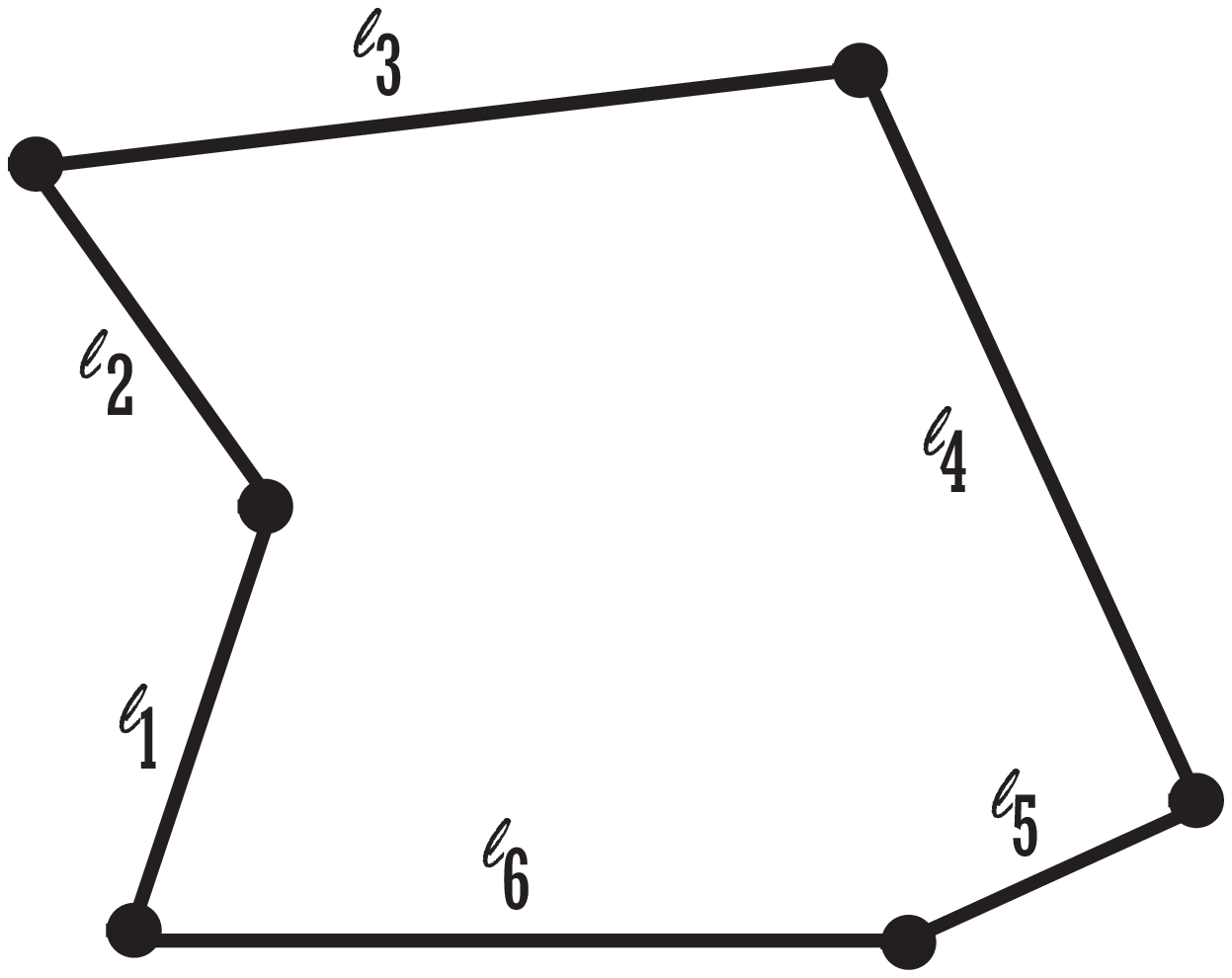}}
\end{center}
\end{figure}
We consider a pair of configurations of the linkage as being identical if one can be obtained from the other by a rigid motion of the space $\R^3$.
The configuration space of the linkage, denoted in this paper by
\begin{eqnarray}
N_\ell = \{(u_1, \dots, u_n)\, \in \, S^2\times \dots\times S^2;
\, \sum_{i=1}^n l_iu_i \, =\, 0\in \R^3\}\, /\, {\rm {SO}}(3),
\end{eqnarray}
parameterizes all
possible configurations. Here $\ell =(l_1, l_2, \dots, l_n)\in \R^n_+$ is the $n$-tuple of the bar lengths, called the {\it length vector}\footnote{In this paper $\R^n_+\subset \R^n$ denotes the set of all points $(l_1, \dots, l_n)$ having positive coordinates $l_1>0, \dots, l_n>0$.}. Spaces $N_\ell$ are also known as {\it polygon spaces} as they parameterize shapes of all $n$-gons in $\R^3$ having sides of length $l_1, \dots, l_n$.

The spaces $N_\ell$ appear in molecular biology where they represent shapes of
long molecules. Clearly, information about topological properties of these spaces
may lead to interesting new effects in molecular and chemical design.
Statistical shape theory, see e.g. \cite{Kendall}, is another area where the spaces
$N_\ell$ play an interesting role: they describe the space of shapes having
certain geometric properties with respect to the central point.
Having in mind these applications it is quite natural to assume that the number of links $n$ is large, $n\to \infty$, and that the numbers $l_i>0$
are not entirely known or are known with some random error.

Let us now give some basic facts concerning the topology of $N_\ell$
and its dependence on the length vector $\ell\in \R^3_+$.
For a {\it generic} $\ell$ (this term is explained below), $N_\ell$ is a compact smooth manifold of dimension $2(n-3)$.
If $\ell$ is not generic then
$N_\ell$ is a compact manifold with singularities.
Clearly, $N_\ell = N_{t\ell}$ for any $t>0.$
Hence we may consider $\ell$ as lying in the quotient space $\Delta^{n-1}$ of $\R^n_+$ modulo the action of $\R_+$ by scalar multiplication.
Clearly, $\Delta^{n-1}$ can be identified with the interior of the standard
simplex, i.e. the set given by the inequalities $l_1 >0, \dots,
l_n>0$ and $\sum l_i =1$.

It is easy to see that $N_\ell$ is diffeomorphic to $N_{\ell'}$ if
$\ell'$ is obtained from $\ell$ by permuting coordinates. Let $\Sigma_n$ denote the permutation group of $n$ symbols. Clearly $\Sigma_n$ acts on $\R^n$ and on $\Delta^{n-1}$ permuting coordinates and the manifold $N_\ell$ depends only on the $\Sigma_n$-orbit of the vector $\ell$.

To explain further the character of dependence of $N_\ell$ on $\ell$ we need to introduce the concept of a {\it chamber}.
For any subset $J\subset \{1, \dots, n\}$
we denote by $H_J\subset \R^n$ the hyperplane defined by the
equation
\begin{eqnarray}\label{wall}\sum_{i\in J} l_i = \sum_{i\notin J}
l_i.\end{eqnarray}
The connected components of the complements $$\R^{n}_+ - \bigcup_J H_J$$ are called {\it chambers}.
{\it Generic} length vectors are defined as those lying in chambers,  not on hyperplanes $H_J$.
It is an interesting combinatorial problem to find or estimate the number $c_n$ of different
$\Sigma_n$-orbits of chambers for large $n$.  The numbers $c_n$ are known for $n\le 9$, see \cite{HR}:
\vskip 0.5cm
\begin{center}
\begin{tabular}{|c|c|c|c|c|c|c|c|}
\hline
$n$ &\footnotesize 3&\footnotesize 4&\footnotesize 5&\
\footnotesize 6 &\footnotesize 7&\footnotesize 8&\footnotesize 9
\\[1mm] \hline 
\small $c_n$ &\footnotesize 2 &\footnotesize 3&
\footnotesize 7&\footnotesize 21&\footnotesize 135&
\footnotesize 2470&  \footnotesize  175428
\\[1mm] \hline
\end{tabular}
\end{center}
\vskip 0.2cm
We see that the number $c_n$ grows very fast.
The diffeomorphism type of $N_\ell$ is constant when
$\ell$ varies inside a chamber. One of the main results of \cite{FHS} states that for a given $n$
{\it the map $\ell\mapsto N_\ell$ gives a one-to-one correspondence between the $\Sigma_n$-orbits of chambers and the diffeomorphism types of
manifolds $N_\ell$ where $\ell \in \R^n_+$ is generic}.

The following picture summarizes our description of {\it the field of topological spaces} $\ell\mapsto N_\ell$ viewed as a single object.
The open simplex $\Delta^{n-1}$ is divided into a huge number of tiny chambers, each representing a diffeomorphism type of manifolds $N_\ell$.
The
symmetric group $\Sigma_n$ acts on the simplex $\Delta^{n-1}$ mapping chambers to chambers and manifolds $N_\ell$ and $N_{\ell'}$
are diffeomorphic if and only if the vectors $\ell, \ell'$ lie in chambers belonging the same $\Sigma_n$-orbit.

The main idea of this work is to use methods of probability theory and statistics in dealing with the
variety of diffeomorphism types of configuration spaces $N_\ell$ for $n$ large.
In applications different manifolds $N_\ell$ appear with different
probabilities and our intention is to study the most \lq\lq frequently emerging\rq\rq\, manifolds $N_\ell$ and the mathematical expectations of their
topological invariants.
Formally, we view the length vector $\ell\in \Delta^{n-1}$ as a random variable
whose statistical behavior is characterized by a probability measure $\nu_n$. Topological invariants of $N_\ell$ become random functions and
their mathematical expectations might be very useful for applications. Thus, one is led to study the {\it average} or {\it expected} Betti numbers\footnote{It is well known that all odd-dimensional Betti numbers of $N_\ell$ vanish, see \cite{Kl}.}
\begin{eqnarray}\label{avbn}
{\rm {E}}(b_{2p}(N_\ell)) = \int_{\Delta^{n-1}}b_{2p}(N_\ell)d\nu_n\end{eqnarray}
where the integration is understood with respect to $\ell$.
One of the main results of this paper states that for $p$ fixed and $n$ large this average $2p$-dimensional Betti number
can be calculated explicitly up to an exponentially small error.
More precisely, we prove that
$${\rm {E}}(b_{2p}(N_\ell)) = \int_{\Delta^{n-1}}b_{2p}(N_\ell)d\nu_n \, \sim\,  \sum_{i=0}^p \binom{n-1}{i}. $$
It might appear surprising that the asymptotic value of the average Betti number $b_{2p}(N_\ell)$
does not depend of the sequence of probability measures $\nu_n$ which are allowed to vary in an ample
class of {\it admissible} probability measures described in \S \ref{results} below.

We also find the asymptotics of the average Betti numbers $b_p(M_\ell)$ of configuration spaces of {\it planar} polygon spaces
\begin{eqnarray}
M_\ell = \{(u_1, \dots, u_n)\, \in \, S^1\times \dots\times S^1;
\, \sum_{i=1}^n l_iu_i \, =\, 0\in \R^2\}\, /\, {\rm {SO}}(2).
\end{eqnarray}
In paper \cite{FK} we calculated the asymptotic values of the average Betti numbers
\begin{eqnarray}\label{avml}
{\rm {E}}(b_p(M_\ell)) = \int_{\Delta^{n-1}} b_p(M_\ell)d\nu_n \end{eqnarray}
for two sequences of probability measures $\nu_n$ on $\R^n_+$.
It was discovered in \cite{FK} that
for large $n$ the answers for
these two distinct measures were equal.
The present paper explains this {\it universality phenomenon}.
We also compute asymptotics of the moments
\begin{eqnarray}\label{moments}
\int_{\Delta^{n-1}}b_{2p}(N_\ell)^kd\nu_n, \quad \quad \int_{\Delta^{n-1}}b_p(M_\ell)^kd\nu_n,\end{eqnarray} where $k=1, 2, 3, \dots$ assuming that $n$
tends to infinity.

In this paper we employ a method different from the one used in \cite{FK}: instead of dealing with explicit expressions for Betti numbers we specify a domain in the simplex of parameters where the behavior of Betti numbers can easily be described and, moreover, the volume of the complement of the domain is exponentially small.
The proofs presented below are shorter than in \cite{FK} although
theorems of the present paper are more general in several respects: (a) they allow more general class of measures, (b) in this paper we treat both cases of
planar and spatial linkages and (c) the present approach is applicable to higher moments as well.

Here are a few comments on the history of the problem.
Polygon spaces were studied by K. Walker \cite{Wa},
M. Kapovich and J. Millson \cite{KM1} and others.
Betti numbers of the spaces $N_\ell$ were first described by A.A. Klyachko \cite{Kl}
who used methods of algebraic geometry.
 J.-Cl. Hausmann and A.
Knutson \cite{HK} applied methods of symplectic topology (toric varieties) to study the cohomology algebras $H^\ast(N_\ell)$.
Betti numbers of planar polygon spaces $M_\ell$ as functions of the length vector $\ell$ were found in \cite{FS}; this result uses techniques of
Morse theory of manifolds with involution. The recent preprint \cite{FHS} gives a general classification of diffeomorphism types of polygon spaces $M_\ell$ and $N_\ell$ in terms of combinatorics of chambers and the action of the symmetric group $\Sigma_n$.

\section{Statements of the main results}\label{results}

To state the main results of this paper we need to define what is meant by an {\it admissible sequence of measures}.

For a vector $\ell=(l_1, \dots, l_n)$ we denote by
$|\ell| = \max\{|l_1|, \dots, |l_n|\}$
the maximum of absolute values of coordinates.
The symbol $\Delta^{n-1}$ denotes the open unit simplex, i.e. the set of all vectors $\ell=(l_1, l_2, \dots, l_n)\in \R^n$ such that $l_i>0$ and $l_1+\dots+l_n=1$.
Let $\mu_n$ denote the Lebesgue measure on $\Delta^{n-1}$
normalized so that $\mu_n(\Delta^{n-1}) =1.$ In other words, for a  Lebesgue measurable subset $A\subset \Delta^{n-1}$ one has
$$\mu_n(A) = \frac{\vol(A)}{\vol (\Delta^{n-1})}$$ where the symbol $\vol $ denotes the $(n-1)$-dimensional volume.

For an integer $p\geq 1$ we denote by
\begin{eqnarray}\Lambda_p= \Lambda_p^{n-1} =\{\ell\in \Delta^{n-1}; |\ell|\ge (2p)^{-1}\}.
\end{eqnarray}
Clearly, $\Lambda_p\subset \Lambda_q$ for $p\le q$ and $\Lambda_p = \Delta^{n-1}$ for $2p \geq n$.

It will be helpful to think of $p$ being fixed and of $n$ being large, say, tending to $\infty$.
The set $\Lambda_p$ is shown on the picture on the left. It is obtained from the simplex $\Delta^{n-1}$ by removing $n$ domains $\Lambda_p^i$ defined as
$\Lambda_p^i=\{\ell=(l_1, \dots, l_n)\in \Delta^{n-1}; l_i\geq (2p)^{-1}\}$. Here $i=1, \dots, n-1$.
Each $\Lambda_p^i$ is homothetic to $\Delta^{n-1}$ with coefficient $(1-\frac{1}{2p})$ and hence $\mu_n(\Lambda_p^i)=(1-\frac{1}{2p})^{n-1}$. It follows that $\vol(\Lambda_p) \leq n(1-\frac{1}{2p})^{n-1}$. We conclude that the Lebesgue measure of the set $\Lambda_p$ is exponentially small for large $n$.
\begin{figure}[h]
\vspace{-10pt}
\begin{center}
\resizebox{10cm}{5cm}
{\includegraphics[24,407][580,704]{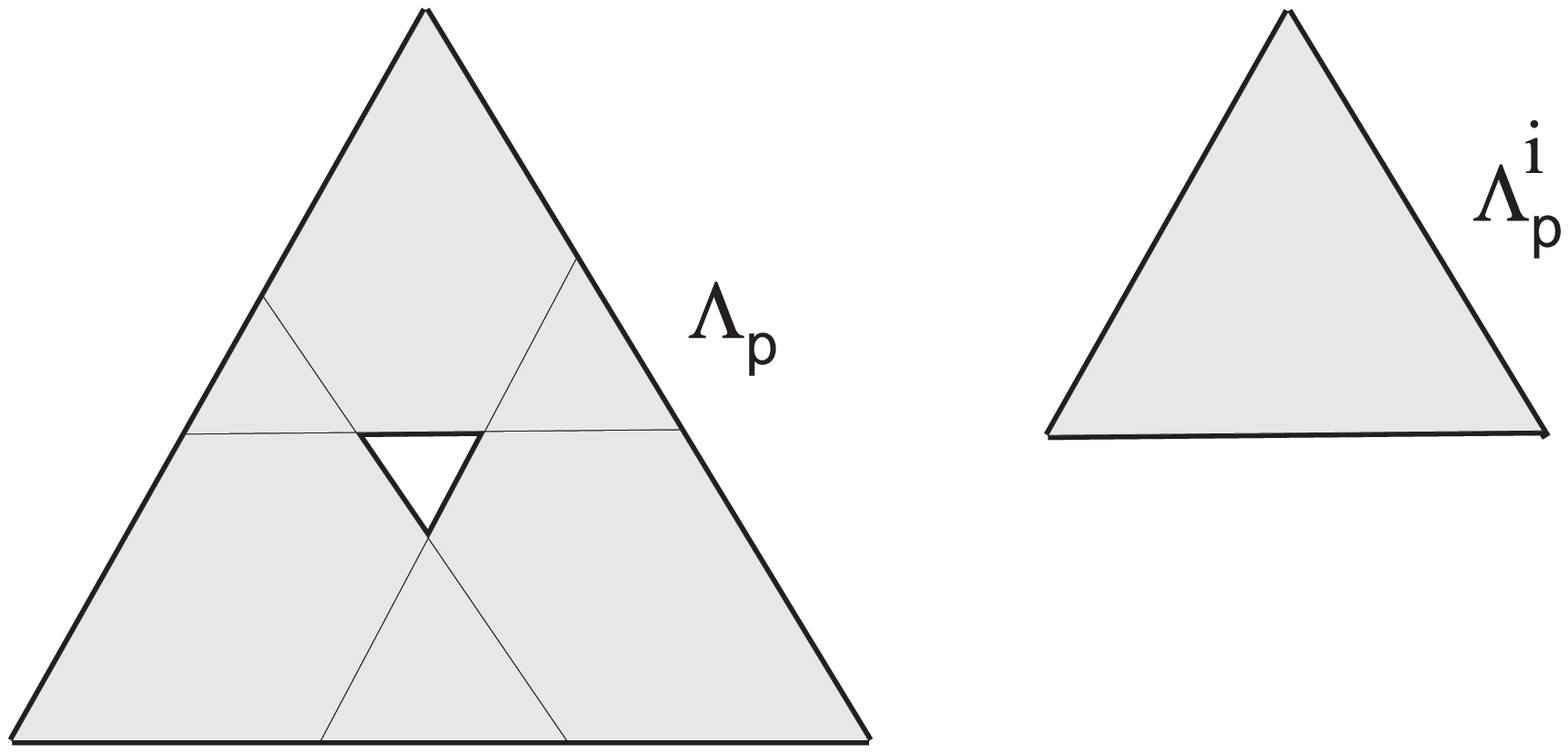}}
\end{center}
\vspace{-20pt}
\end{figure}

\begin{definition}\label{admissible}
{\rm Consider a sequence of probability measures $\nu_n$ on $\Delta^{n-1}$ where $n=1, 2, \dots$. It is called {\it admissible} if
$\nu_n = f_n\cdot \mu_n$ where $f_n: \Delta^{n-1}\to \R$ is a sequence of functions satisfying:
{\rm(i)} $f_n\geq 0$, {\rm(ii)}  $\int_{\Delta^{n-1}} f_n d\mu_n =1$, and
{\rm(iii)}  for any $p\geq 1$ there exist constants $A>0$ and $0<b<2$ such that
\begin{eqnarray}
f_n(\ell) \le A\cdot b^n
\end{eqnarray}
for any $n$ and any $\ell\in \Lambda_p\subset \Delta^{n-1}$.}
\end{definition}
Note that property (iii) imposes restrictions on the behavior of the sequence $\nu_n$ only in domains $\Lambda_p^{n-1}$.

{\bf Example}. Consider the unit cube $\square^n\subset \R^n_+$
given by the inequalities $0~\leq  l_i~\leq~ 1$ for $i=1, \dots, n$.
Let $\chi_n$ be the probability measure on $\R^n_+$
supported on $\square^n\subset \R^n_+$ such that the restriction
$\chi_n|_{\square^n}$ is the Lebesgue measure, $\chi_n(\square^n)=1$.
Consider the sequence of induced measures \begin{eqnarray}\label{induced}\nu_n=q_\ast(\chi_n)\end{eqnarray}
on simplices $\Delta^{n-1}$ where $q: \R_+^n\to \Delta^{n-1}$ is the normalization map
$q(\ell) = t\ell$ where $t=(l_1+\cdots+ l_n)^{-1}.$ The measures $\nu_n$ have a very clear geometric meaning: it is the probability
distribution in the case when
the bar lengths $l_i$ are independent and are
uniformly distributed in the unit interval $[0,1]$, see \cite{FK}. The goal of the following arguments is to show that the sequence (\ref{induced}) is admissible.

It is easy to see that
$$\nu_n=f_n\mu_n$$ where $f_n: \Delta^{n-1}\to \R$ is given by
\begin{eqnarray}\label{stama}
f_n(\ell)= k_n \cdot |\ell|^{-n}, \quad \ell\in \Delta^{n-1}, \quad k_n\in \R.\end{eqnarray}
Indeed, consider coordinates $0\leq y_1, \dots, y_n\leq 1$ in $\square^n$ and another coordinate system in $\square^n$ is given by variables $z_1, \dots, z_{n-1}, t$ satisfying
$z_i\geq 0$, $z_1+ \dots+z_{n-1}\leq 1$ and $0\leq t\leq |z|^{-1}$, where $z=(z_1, \dots, z_n)$ and $z_n$ is an
auxiliary variable given by $z_n=1-z_1-\dots -z_{n-1}$.
These two coordinate systems are related by $y_i=tz_i$, where $i=1, \dots, n$.
A simple calculation shows that
$$dy_1\wedge dy_2 \wedge \dots dy_n = t^{n-1} dz_1\wedge \dots dz_{n-1}\wedge dt$$
leading to formula (\ref{stama}).
The constant $k_n$ which appears in (\ref{stama}) can be found (using (ii) of Definition \ref{admissible}) from the equation
\begin{eqnarray}\label{kn}
k_n^{-1}= \int_{\Delta^{n-1}}|\ell|^{-n}d\mu_n.\end{eqnarray}

If $\ell\in \Lambda^{n-1}_p$ then $f_n(\ell) \leq k_n\cdot (2p)^n$. We can represent $\Lambda_p^{n-1}$ as the union
$\Lambda_p^1\cup \cdots \cup \Lambda_p^n$ where $$\Lambda_p^i=\{(l_1, \dots, l_n)\in \Delta^{n-1}; \, \,  l_i\geq (2p)^{-1}\}, \quad i=1, \dots, n.$$
Clearly, $\mu_n(\Lambda_p^i)= \left(\frac{2p-1}{2p}\right)^{n-1}$ and hence
$\mu_n(\Delta^{n-1} - \Lambda_p) \geq 1 - n\left(\frac{2p-1}{2p}\right)^{n-1}.$
Using (\ref{kn}) we find that
$k_n^{-1} \geq (2p)^n \cdot \left( 1- n \left(\frac{2p-1}{2p}\right)^{n-1}\right)$. This shows that the sequence
$(2p)^nk_n$ remains bounded
as $n\to \infty$ implying (iii) of Definition \ref{admissible}. Hence, the sequence of measures $\{\nu_n\}$ is admissible.

Next we state main theorems of this paper.

\begin{theorem}\label{thmain2} Fix an admissible sequence of probability measures $\nu_n$ and an integer
$p\geq 0$, and consider the $2p$-dimensional Betti number (\ref{avbn}) of polygon spaces $N_\ell$ in $\R^3$ as a random variable on $\Delta^{n-1}$,
 for large $n\to \infty$.
 Then there exist constants $C>0$ and $0<a<1$ (depending on the sequence of measures $\nu_n$ and on the number $p$ but independent of $n$)
such that the average Betti numbers (\ref{avbn}) satisfy
\begin{eqnarray}
\left| \int\limits_{\Delta^{n-1}}b_{2p}(N_\ell)d\nu_n - \sum_{i=0}^p\binom{n-1}{i}\right| < C\cdot a^n
\end{eqnarray}
for all $n$.
\end{theorem}


\begin{theorem}\label{thmain1} Fix an admissible sequence of probability measures $\nu_n$ and an integer
$p\geq 0$, and consider the average $p$-dimensional Betti number (\ref{avml}) of planar polygon spaces
 for large $n\to \infty$.
 Then there exist constants $C>0$ and $0<a<1$ (depending on the sequence of measures $\nu_n$ and on the number $p$ but independent of $n$)
such that
\begin{eqnarray}
\left| \int\limits_{\Delta^{n-1}}b_p(M_\ell)d\nu_n - \binom{n-1}{p}\right| < C\cdot a^n
\end{eqnarray}
for all $n$.
\end{theorem}

\section{Volume of the frustum of a simplex}

In this section we obtain a formula for the volume of the intersection of a simplex with a half-space. General formulae of this kind are well known, see
\cite{CS}, \cite{A}, \cite{Va}, \cite{FK}. However in this paper we have to consider a highly degenerate situation when the functional
determining the half-space  takes only two distinct values of the set of vertices. We give an explicit formula and its proof in this special case.

\begin{lemma}\label{lmrx} Let $b_0, \dots, b_n\in \R^n$ be vertices of a simplex $\Delta^n\subset \R^n$. Let $\phi: \R^n\to \R$ be an affine functional such that
$\phi(b_i) = -1$ for $i=0, \dots, p-1$ and $\phi(b_i)=1$ for $i=p, p+1, \dots, n$. For $x\in \R$ denote by $H_x$ the half-space
$H_x=\{v\in \R^n; \phi(v)\leq x\}$. Then the ratio $$r(x) = \frac{\vol(H_x\cap \Delta)}{\vol(\Delta)}$$
 for $x\in [-1, 1]$ is given by
\begin{eqnarray}\label{rx}
r(x) = \left(\frac{x+1}{2}\right)^{q}\cdot \sum_{k=0}^{p-1}
\, \binom{q-1+k}{q-1}\cdot
\left(\frac{1-x}{2}\right)^k.
\end{eqnarray}
Here $q= n-p+1$ denotes the multiplicity of value $1=\phi(b_i)$.
\end{lemma}
\begin{proof}
The function $r(x)$ is closely related to spline functions which were introduced in \cite{CS}, see \S 3.
Clearly $r(x)$ vanishes for $x\le -1$ and $r(x)$ is identically $1$ for $x\ge 1$.
Moreover, from geometrical considerations
(see \cite{CS}, \S 3)  we know that:

(a) $r(x)$ is a polynomial of degree $n$ for $x\in [-1, 1]$,

(b) $r(x)$ is $C^{n-p}$ near $-1$,

(c) $r(x)$ is $C^{n-q}$ near $1$.

Hence we obtain that \begin{eqnarray}\label{plusone}
r^{(j)}(-1)=0 \quad \mbox{for}\quad j=0, 1, \dots, n-p\end{eqnarray}
and
\begin{eqnarray}\label{minusone}(r(x)-1)^{(j)}(1)=0 \quad\mbox{for}\quad j =0, 1, \dots, n-q.\end{eqnarray}
Requirement (\ref{plusone}) implies that $r(x) = (x+1)^{n-p+1}g(x)$
where $g(x)$ is polynomial of degree $p-1$. To satisfy (\ref{minusone}) we must take for $g(x)$ the sum of terms
 of degree $\leq p-1$
of Taylor expansion at $x=1$ of the rational function $(x+1)^{p-n-1}$.
Computing this Taylor expansion we obtain (\ref{rx}). \end{proof}
\begin{corollary}\label{corrx} Under conditions of Lemma \ref{lmrx} the number $r_{p,q}=r(0)$ is given by
\begin{eqnarray}\label{r0}
r_{p,q} = 2^{-q}\cdot \sum_{k=0}^{p-1}
\, \binom{q-1+k}{k}\cdot  2^{-k}
\end{eqnarray}
\end{corollary}
The number $r_{p,q}$ has a clear geometric meaning. Divide the set of vertices of $\Delta^{n}$ into two subsets, one containing $p$ elements and another containing $q$ elements, here $p+q=n+1$. Let $F$ and $F'$ denote simplices generated by each of these subsets. Any point of $\Delta^{n}$ lies on a unique segment connecting a point of $F$ with a point of $F'$. Then $r_{p,q}$ is the relative volume of the set of points $x\in \Delta^{n}$ such that $x=(1-t)a+ta'$ where $a \in F$, $a'\in F'$ and $t\in [0,1/2]$.

\section{Proof of Theorem \ref{thmain2}}

We start by recalling the description of Betti numbers of polygon spaces $N_\ell$ given by A.A. Klyachko \cite{Kl} and by
 J.-Cl. Hausmann and A. Knutson \cite{HK}.  A subset $J\subset
\{1, \dots, n\}$ is called {\it short} if
\begin{eqnarray*}
\sum_{i\in J}l_i \, < \, \sum_{i\notin J}l_i.\end{eqnarray*} A
subset is called {\it long} if its complement is short.
Given a generic length vector we denote by $\alpha_j(\ell)$ the number of subsets $J\subset \{1, \dots, n-1\}$ such that $|J|=j$ and the set $J\cup \{n\}$ is short.
The result of Corollary 4.3 from \cite{HK} can be equivalently stated as
\begin{eqnarray}\label{bettin}
b_{2p}(N_\ell) = \sum_{j=0}^p \left[\alpha_j(\ell) - \alpha_{n-j-2}(\ell)\right].
\end{eqnarray}

We will need the following upper bound.

\begin{proposition}\label{upperbound1}
For a generic length vector $\ell=(l_1, \dots, l_n)\in \Delta^{n-1}$ with $n\geq 3$ one has \begin{eqnarray}b_{2p}(N_\ell) \leq 2\cdot (n-1)^p.\end{eqnarray}
\end{proposition}
\begin{proof} From (\ref{bettin}) one obtains
\begin{eqnarray*}
b_{2p}(N_\ell) \leq \sum_{j=0}^p \alpha_j(\ell) \le
\sum_{j=0}^p \binom{n-1}{j}
\\ \\
\le \sum_{j=0}^p (n-1)^j = \frac{(n-1)^{p+1}-1}{n-2} \le \frac{(n-1)^{p+1}}{n-2} \le 2\cdot (n-1)^p.
\end{eqnarray*}
On the last step we used the assumption $n\geq 3$.
\end{proof}
Given an integer $p$ we denote by
\begin{eqnarray}
\Gamma_p=\Gamma_p^{n-1}\subset \Delta^{n-1}
\end{eqnarray}
the set of all length vectors $\ell=(l_1, \dots, l_n) \in \Delta^{n-1}$
such that any subset $J\subset \{1, \dots, n\}$ of cardinality
$|J|=p$ is short with respect to $\ell$. Clearly $\Gamma_p\supset \Gamma_{q}$ if $p\le q$ and, moreover, $\Gamma_p=\emptyset$ for $p\ge n/2$.

\begin{proposition}\label{constant1} If $\ell\in \Gamma_p\subset \Delta^{n-1}$ is generic then for all $q\leq p-2$ one has
\begin{eqnarray}\label{bnl}
b_{2q}(N_\ell) = \sum_{j=0}^q
\binom{n-1}{j}.
\end{eqnarray}
\end{proposition}
\begin{proof}
If $\ell\in \Gamma_p$ and $j\leq p-1$ then $\alpha_j(\ell) =\binom{n-1}{j}$ since any subset of cardinality $j+1\leq p$ is short with respect to $\ell$.
Similarly, for $j\le p-2$ one has $\alpha_{n-2-j}(\ell)=0$. Indeed, $\alpha_{n-2-j}(\ell)$ coincides with the number of long subsets $J\subset \{1, \dots, n-1\}$
of cardinality $j+2$ and (as follows from our assumption $j+2\leq p$) all such subsets are short. Formula (\ref{bnl}) now follows from (\ref{bettin}).
\end{proof}

Our next goal is to show that for a fixed $p$ the Lebesgue measure of
the set $\Gamma_p$ is exponentially close to the measure of the whole
simplex $\Delta^{n-1}$ as $n\to \infty$. More precisely, we will prove the
following statement:
\begin{proposition}\label{prop3} One has
\begin{eqnarray}\label{ineq3}
1- n^{2p} \cdot 2^{-n}\, \le \, \, \frac{\vol(\Gamma_p)}{\vol(\Delta^{n-1})}\, \, \le \,  1
\end{eqnarray}
\end{proposition}
\begin{proof}
Let $J\subset \{1, \dots, n\}$ be a subset with $|J|=p$. Let $L\subset \R^n$ be the affine subspace given by the equation $l_1+\dots+l_n=1$. Denote by
$\phi_J: L\to \R$ the affine functional
\begin{eqnarray}\label{phij}
\phi_J(l_1, \dots, l_n) = \sum_{i\in J}l_i - \sum_{i\notin
J}l_i.
\end{eqnarray}
Denote by $V_J\subset \Delta^{n-1}$ the set of vectors
$\ell\in \Delta^{n-1}$ satisfying the inequality $\phi_J(\ell)
\geq 0$. It is clear that
\begin{eqnarray}\label{intermediate}
\Delta^{n-1} - \Gamma_p\, =\, \bigcup_{|J|=p} V_J
\end{eqnarray} and therefore
\begin{eqnarray}\label{equal3}
 \frac{\vol(\Delta^{n-1} - \Gamma_p)}{\vol(\Delta^{n-1})} \, \leq\,
\sum\limits_{|J|=p} \frac{\vol(V_J)}{\vol(\Delta^{n-1})}
\end{eqnarray}
Inequality (\ref{ineq3}) now follows by combining (\ref{equal3}) with the inequality of Proposition \ref{prop4} (see below)
 which allows estimating each term of (\ref{equal3}). The number of terms in the sum (\ref{equal3}) equals $\binom{n}{p}$ which is less or equal than $n^p$.
\end{proof}
\begin{proposition}\label{prop4} Assume that $n> 4$. Then for any subset $J\subset \{1, \dots, n\}$ with $|J|=p$ one has
\begin{eqnarray}\label{estim}\frac{\vol(V_J)}{\vol(\Delta^{n-1})} < n^{p}\cdot
2^{-n}.\end{eqnarray}
\end{proposition}
\begin{proof} Consider the complement $\bar J$ of $J$ and the functional $\phi_{\bar J}= - \phi_J: L\to \R$, see (\ref{phij}). Clearly
$V_J$ can be represented as
the intersection $H\cap \Delta^{n-1}$ where $H$ is the halfspace $H=\{\ell\in L; \phi_{\bar J}(\ell) \leq 0\}$.
Comparing with notation of Lemma \ref{lmrx} and Corollary \ref{corrx} we may write
\begin{eqnarray}\label{inter}
\frac{\vol(V_J)}{\vol(\Delta^{n-1})} = r_{p,n-p}=  2^{-n}\cdot \sum_{k=0}^{p-1}
\, \binom{n-p-1+k}{k}
\cdot  2^{p-k} \end{eqnarray}
The binomial coefficients can be estimated by $ \binom{n-p-1+k}{k}
\le n^k
$
and therefore
$$
\frac{\vol(V_J)}{\vol(\Delta^{n-1})} \leq 2^{p-n}\cdot \sum_{k=0}^{p-1} \left(\frac{n}{2}\right)^k \leq 2^{p-n}\cdot \frac{\left(\frac{n}{2}\right)^p}{\frac{n}{2}-1} < n^p\cdot 2^{-n}.
$$
On the last step we used the assumption $n> 4$.
This completes the proof of Proposition \ref{prop4}.
\end{proof}
\begin{proof}[Proof of Theorem \ref{thmain2}] Let $\nu_n$ be an admissible sequence of measures on $\Delta^{n-1}$, where $n=1, 2, \dots$, see Definition
\ref{admissible}. Using Proposition \ref{constant1} we obtain
\begin{eqnarray*}
\left| \int_{\Delta^{n-1}} b_{2p}(N_\ell)d\nu_n - \sum_{i=0}^p \binom{n-1}{i}
\right| = \\ \\
\left| \int\limits_{\Delta^{n-1}-\Gamma_{p+2}} b_{2p}(N_\ell)d\nu_n - \sum_{i=0}^p \binom{n-1}{i}\cdot \nu_n(\Delta^{n-1}-\Gamma_{p+2})\right|  \\ \\ \le
\left[ n^{p+1} +  \sum_{i=0}^p \binom{n-1}{i}\right] \cdot \nu_n(\Delta^{n-1}-\Gamma_{p+2}).
\end{eqnarray*}
The last inequality uses Proposition \ref{upperbound1}. If $\ell=(l_1, \dots, l_n) \in \Delta^{n-1}-\Gamma_{p+2}$
then there exists a subset $J\subset \{1, \dots, n\}$ with $|J|=p+2$ such that $\sum_{i\in J} l_i \geq 1/2$ and hence
$l_i\geq 1/(2p+4)$ for some $i\in J$. This shows that $\Delta^{n-1}-\Gamma_{p+2}\subset \Lambda_{p+2}$ (which was defined before Definition
\ref{admissible} in \S \ref{results}. Therefore we may use property (iii) from Definition \ref{admissible} to continue the previous string of inequalities
\begin{eqnarray*}
\left[ n^{p+1} +  \sum_{i=0}^p\binom{n-1}{i}\right] \cdot \nu_n(\Delta^{n-1}-\Gamma_{p+2})
\le 2n^{p+1} \cdot \mu_n(\Delta^{n-1}-\Gamma_{p+2})\cdot M_n
\end{eqnarray*}
where
\begin{eqnarray*}M_n=\max_{\ell\in \Lambda_{p+2}} f_n(\ell)\le Ab^n.\end{eqnarray*}
The constants $A>0$ and $0<b<2$ appearing here are given by Definition \ref{admissible}, see (iii).  Using Proposition \ref{prop3} we have
$\mu_n(\Delta^{n-1}-\Gamma_{p+2})\le n^{2p+4}\cdot 2^{-n}$. Hence we finally obtain
\begin{eqnarray*}
\left| \int_{\Delta^{n-1}} b_{2p}(N_\ell)d\nu_n - \sum_{i=0}^p \binom{n-1}{i}
\right|
\le 2n^{p+1} \cdot Ab^n \cdot n^{2(p+2)} \cdot 2^{-n}\\ \\
\le 2A\cdot n^{3p+5}\cdot \left(\frac{b}{2}\right)^n \le Ca^n.
\end{eqnarray*}
Here $a$ is any number satisfying $b/2<a<1$ and $C>0$ is chosen accordingly.
\end{proof}

\section{Proof of Theorem \ref{thmain1}}

First we recall the result of \cite{FS} describing Betti numbers $b_p(M_\ell)$ of planar polygon spaces as functions of the length vector $\ell$.
 Fix an index $1\leq i\leq n$ such that $l_i$ is maximal among
$l_1, \dots, l_n$. Denote by $a_p(\ell)$ the number of short
subsets $J\subset \{1, \dots, n\}$ of cardinality $|J|=1+p$
containing $i$.
A subset
$J\subset \{1, \dots, n\}$ is called {\it median} if
\begin{eqnarray*}
\sum_{i\in J} l_i=\sum_{i\notin J} l_i.\end{eqnarray*}
Denote by $\tilde a_p(\ell)$ the number of median
subsets $J\subset \{1, \dots, n\}$ containing $i$ and such that
$|J|=1+p$. Then one has
\begin{eqnarray}\label{fs}
b_p(M_\ell) = a_p(\ell) + \tilde a_p(\ell) + a_{n-3-p}(\ell).
\end{eqnarray}
 for $p=0, 1, \dots,
n-3$, see  \cite{FS}.

\begin{proposition}\label{upperbound} For a generic length vector $\ell=(l_1, \dots, l_n)\in \Delta^{n-1}$ one has \begin{eqnarray}b_p(M_\ell) \leq n^{p+2}.\end{eqnarray}
\end{proposition}
\begin{proof} From the description of $b_p(M_\ell)$ given above one deduces that for a generic $\ell$ the number $\tilde a_p(\ell)$ vanishes and
$a_p(\ell)$ (see (\ref{fs})) is less or equal than the binomial coefficient $\left(\begin{array}{c}  n-1\\ p\end{array} \right)$.
Hence using (\ref{fs}) we obtain
$$b_p(M_\ell) \le \left(\begin{array}{c}  n-1\\ p\end{array} \right) + \left(\begin{array}{c}  n-1\\ n-3- p\end{array} \right) \le (n-1)^p + (n-1)^{p+2}\le n^{p+2}.$$ This completes the proof.
\end{proof}

\begin{proposition}\label{constant} If $\ell\in \Gamma_p\subset \Delta^{n-1}$ is generic then \begin{eqnarray}
b_j(M_\ell)= \left(\begin{array}{c}  n-1\\ j\end{array} \right)\quad \mbox{for all}\quad j \le p-2.\end{eqnarray}
\end{proposition}
\begin{proof}
As follows from (\ref{fs}), if $\ell$ is generic and lies in $\Gamma_p$ then $\tilde a_j(\ell)=0$ and for $j\le p$ the number $a_j(\ell)$ equals the number of all subsets of $\{1, \dots, n\}$ which do not contain the element with the maximal length $l_i$, i.e.  $a_j(\ell) =\left(\begin{array}{c}  n-1\\ j\end{array} \right)$.
The number $a_{n-3-j}(\ell)$ is the number of all short subsets of cardinality $n-2-j$ containing the maximal element; the complements of these sets are short and their cardinality is $j+2$. However if $j+2\le p$ all such subsets must be short, i.e. $a_{n-3-j}(\ell) =0$ for $j\le p-2$.
\end{proof}

\begin{proof}[Proof of Theorem \ref{thmain1}] The proof essentially repeats the arguments of the proof of Theorem \ref{thmain2} with Proposition \ref{upperbound} replacing Proposition \ref{upperbound1} and with Proposition \ref{constant} replacing Proposition \ref{constant1}.
\end{proof}

\section{Normal length vectors}

In paper \cite{FHS} we introduced the notion of a normal length vector. A vector $\ell=(l_1, \dots, l_n)\in \R^n_+$ is called {\it normal} if the intersection of all
subsets $J\subset \{1, \dots, n\}$ of cardinality 3 which are long with respect to $\ell$ is not empty.
A length vector $\ell$ with the property that all subsets of cardinality $3$ are short with respect to $\ell$
is normal
 since then the intersection of all long subsets of cardinality $3$
equals $\{1, \dots, n\}$ as the intersection of the empty family.

The
importance of normal length vectors stems from the following result proven in \cite{FHS}:

\begin{theorem}[\cite{FHS}]
Suppose that $\ell, \ell'\in \Delta^{n-1}$ are two ordered length
vectors such that there exists a graded algebra isomorphism
between the integral cohomology algebras $H^\ast(M_\ell) \to
H^\ast(M_{\ell'})$. Assume that one of the vectors $\ell,
\ell'$ is normal. Then the other vector is normal as well and
$\ell$ and $\ell'$ lie in the same
stratum of the simplex $\Delta^{n-1}$; in particular, the polygon spaces $M_\ell$ and $M_{\ell'}$ are diffeomorphic.
\end{theorem}

We apply the technique developed in this paper to show that the Lebesgue measure of the set of
length vectors which are not normal is exponentially small for large $n\to \infty$.

\begin{proposition} Let ${\mathcal N}_n\subset \Delta^{n-1}$ denote the set of all normal length vectors.
The relative volume of $\mathcal N_n$ satisfies the following inequality
\begin{eqnarray}
 \frac{\vol(\Delta^{n-1} - {\mathcal N}_n)}{\vol(\Delta^{n-1})} \, <\,  \frac{24 n^6}{2^n}.
 \end{eqnarray}
\end{proposition}
\begin{proof}We observe that $\Gamma_3 \subset \mathcal N_n$ and applying (\ref{ineq3}) we find
 $$\frac{\vol(\Delta^{n-1} - {\mathcal N}_n)}{\vol(\Delta^{n-1})} \, <\,   \frac{\vol(\Delta^{n-1} - \Gamma_3)}{\vol(\Delta^{n-1})} \, <\,  \frac{24 n^6}{2^n}.$$
\end{proof}

\section{Mathematical expectations of higher moments}

\begin{theorem} Given an admissible sequence of probability measures $\nu_n$ and integers
$p\geq 0$, and $k$, then there exist constants $C>0$ and $0<a<1$
such that the $k$-th powers of the average Betti numbers (\ref{moments}) satisfy
\begin{eqnarray}\label{one1}
\left| \int_{\Delta^{n-1}}b_{2p}(N_\ell)^k d\nu_n - \left(\sum_{i=0}^p\left( \begin{array}{c} n-1\\ i\end{array}\right)\right)^k\right| < C\cdot a^n
\end{eqnarray}
and
\begin{eqnarray}\label{two2}
\left| \int_{\Delta^{n-1}}b_{p}(M_\ell)^k d\nu_n - \left( \begin{array}{c} n-1\\ i\end{array}\right)^k\right| < C\cdot a^n
\end{eqnarray}
for all $n=3, 4, \dots$.
\end{theorem}
\begin{proof} One simply repeats the arguments used in proofs of Theorems \ref{thmain2} and \ref{thmain1} with minor modifications. The scheme remains the same: on the central domain $\Gamma_{p+2}\subset \Delta^{n-1}$ the functions appearing in (\ref{one1}) and (\ref{two2}) coincide. The volume of
the remaining part $\Delta^{n-1}-\Gamma_{p+2}$ is exponentially small (by Proposition \ref{prop3})
and the functions involved have polynomial upper bounds in $n$.
\end{proof}

\section{Some open questions}

In this section I would like to mention a few interesting open questions. The tools of the present paper seem to be inadequate to give their solutions.

It would be interesting to find the average total Betti number
$$B(n,\mu) = \int_{\R^n_+}B(M_\ell)d\mu, \quad \mbox{where}\quad B(M_\ell)=\sum_{p=0}^{n-3}b_p(M_\ell)$$
for various natural probability measures $\mu$ on the spaces of parameters $\R^n_+$ (or on $\Delta^{n-1}$) and
to examine the behavior of $B(n,\mu)$ for $n\to \infty$. Although we know the behavior of the individual average Betti numbers $\int_{\R^n_+} b_p(M_\ell)d\mu$
for large $n$ and for fixed $p$ one cannot simply add the terms up.

By Theorem 2 from \cite{FS} one has an upper bound
\begin{eqnarray}\label{upper}
B(n, \mu) \, \leq\,  2^{n-1}- \left(  \begin{array}{c} n-1\\ r\end{array}  \right) \, \sim\,  2^{n-1}\cdot\left(1-\sqrt{\frac{2}{n\pi}}\right),
\end{eqnarray}
where $r=[(n-1)/2]$ denotes the integer part of $(n-1)/2$. It is plausible that  the RHS of (\ref{upper}) gives the right asymptotic for $B(n,\mu)$.

One can raise a similar question concerning the average total Betti numbers of spatial polygon spaces $N_\ell$.

A homotopy invariant $\tc(X)$, introduced in \cite{F}, measures the complexity of the problem of navigation in a topological space $X$,
viewed as the configuration space of a mechanical system.
It is a challenging problem to compute $\tc(M_\ell)$ as a function of the length vector $\ell\in \R^n_+$ and then study its average
\begin{eqnarray}\label{tcml}
\tc(n, \mu) = \int_{\R^n_+}\, \tc(M_\ell)d\mu
\end{eqnarray}
and behavior as $n\to \infty$ under different assumptions on the measure $\mu$. Formally
the invariant $\tc(X)$ is defined only when $X$ is path-connected. This assumption may be violated in the case of spaces $M_\ell$.
If $X$ is not path-connected it is natural to define $\tc(X)$ as $\max \tc(X_i)$ where $X_i$ are path-connected components of $X$.
If $M_\ell$ is disconnected then it is disjoint union of two tori $T^{n-3}$ (see \cite{KM1}) and hence in this case $\tc(M_\ell)=\tc(T^{n-3}) =n-2$, see \cite{F}.

The similar question concerning spatial polygon spaces $N_\ell$ is
much easier. One can show that $\tc(N_\ell) = 2n-5$ assuming that $N_\ell\not=\emptyset$ and hence
\begin{eqnarray}\label{tcnl}
\int_{\R^n_+}\tc(N_\ell)d\nu_n \, \sim \, 2n-5
\end{eqnarray}
for any admissible sequence of probability measures $\nu_n$. The error in (\ref{tcnl}) is exponentially small for large $n$.

The results concerning average topological complexity (\ref{tcml}), (\ref{tcnl})
of polygon spaces $M_\ell$ and $N_\ell$ may have important applications in
molecular biology, statistical shape theory and robotics.

\end{document}